\documentclass[reqno]{amsart}
\usepackage{amsthm, amssymb, amsfonts}
\usepackage{lmodern}
\usepackage{color}
\usepackage{hyperref}
\usepackage[T5]{fontenc}
\usepackage{amscd,amssymb}
\usepackage[v2,cmtip]{xy}
\usepackage{mathrsfs}
\theoremstyle{plain}
\newtheorem{thm}{Theorem}[section] 
\newtheorem{prop}[thm]{Proposition}

\newtheorem{corl}[thm]{Corollary}
\theoremstyle{definition}

\begin{document} 
	
\title[The primitive Milnor operations on the Dickson invariants]{A note on the action of the primitive Milnor operations on the Dickson invariants}

\author{Nguy\~\ecircumflex n Sum}
\address{Department of Mathematics and Applications, S\`ai G\`on University,
273 An D\uhorn \ohorn ng V\uhorn \ohorn ng, District 5, H\`\ocircumflex\ Ch\'i Minh city, Viet Nam}
\email{nguyensum@sgu.edu.vn}
\subjclass[2010]{Primary 55S10; Secondary 55S05}
\keywords{Polynomial algebra, cohomology operations, modular invariants}
 	 
\maketitle

\begin{abstract} 
In this paper, we present a formula for the action of the primitive Milnor operations on generators of algebra of invariants of the general linear group ${GL_n = GL(n,\mathbb F_p)}$ in the polynomial algebra $P_n= \mathbb F_p[x_1,x_2,\ldots,x_n]$ with $p$ an odd prime number. 
\end{abstract}

\section{Introduction}

Let $p$ be a prime number. Denote by $GL_n = GL(n,\mathbb F_p)$ the general linear  group over the prime field $\mathbb F_p$ of $p$ elements. This group acts on the polynomial $P_n= \mathbb F_p[x_1,x_2,\ldots,x_n]$ in the usual manner. We grade $P_n$ by assigning $\dim x_j= 1$ for $p = 2$ and $\dim x_j= 2$ for $p > 2$. 
Dickson showed in \cite{dic} that the invariant algebra $P_n^{GL_n}$ is a polynomial algebra generated by invariants $Q_{n,s}$, $0\leqslant s<n$, which are called the Dickson invariants.

Let $\mathcal{A}(p)$ be the mod $p$ Steenrod algebra and denote by
$St^{R} \in \mathcal{A}(p)$ the Milnor operation of type $R$, where $R$ is a finite sequence of non-negative integers (see Milnor \cite{mil}, M\`ui \cite{mu1,mui}). For $R = (k)$, $St^{(k)}$ is the Steenrod operation $P^k$. For $\Delta_i = (0,\ldots,0,1)$ of length $i$, $St^{\Delta_i}$ is the primitive Milnor operation in $\mathcal A(p)$. This operation was denoted by $Q^i$ in Adams and Wilkerson \cite{adw}.

The Steenrod algebra  $\mathcal{A}(p)$ acts on $P_n$ by means of the Cartan formula together with the relations $\beta(x_j) = 0$ and 
\begin{align*}
P^k(x_j) &= \begin{cases} x_j, &\mbox{if } k=0,\\  x_j^p, &\mbox{if } k=1,\\ 0, &\mbox{otherwise,}
\end{cases} 
\end{align*}
for $j = 1,\ \! 2, \ldots, n$ (see Steenrod and Epstein \cite{13}). Note that $P^k$ is the Steenrod square $Sq^k$ for $p = 2$, and $\beta$ is the Bockstein operation for $p > 2$.
Since this action commutes with  the one of $GL_n$, it induces an inherited action of $\mathcal{A}(p)$ on $P_n^{GL_n}$. 

The action of the Milnor operations on the modular invariants of linear groups has partially been studied by Smith and Switzer \cite{ssw}, Wilkerson \cite{wil} and the present author \cite{sum,su0,su1,sub,su2}. 

The purpose of the paper is to present a new formula for the action of the primitive Milnor operations $St^{\Delta_i}$ on the Dickson invariants. 

\section{Main Result}

First of all, we introduce some notations. Let  $(e_{1 },\ldots,e_n)$ be a sequence of non-negative integers. Following Dickson \cite{dic}, we define  
$$ [e_{1},e_2,  \ldots,  e_n]  =
\begin{vmatrix} 
x_1^{p^{e_{1}}}&\cdots &x_n^{p^{e_{1}}}\\
x_1^{p^{e_{2}}}&\cdots &x_n^{p^{e_{2}}}\\
\vdots&\cdots  &\vdots\\
x_1^{p^{e_n}} & \cdots & x_n^{p^{e_n}}
\end{vmatrix}. $$

Denote $L_{n,s}=[0,1,\ldots,\hat s,\ldots,n], \, 0\leqslant s\leqslant n,\ 
L_n =L_{n,n}=[0,1,\ldots,n-1].$ Each $[e_{1},e_2,\ldots,e_n]$ is divisible by $L_n$ and $[e_{1},e_2,  \ldots,  e_n]/L_n$ is an $GL_n$-invariant. Then, Dickson invariants $Q_{n,s}$ are defined by
\begin{align*} 
Q_{n,s} &= L_{n,s}/L_n, \ 0 \leqslant s < n.
\end{align*}
By convention, $Q_{n,s} = 0$ for $s < 0$. Note that $Q_{n,0} = L_n^{p-1}$. 
\begin{thm}[See Dickson \cite{dic}] $P_n^{GL_n} = \mathbb F_p[Q_{n,0},Q_{n,1},\ldots,Q_{n,n-1}].$
\end{thm}
 The main result of the paper is the following.
 
\begin{thm}\label{dlc} For any $0\leqslant s < n$ and $ i \geqslant 1$, we have
\begin{align}\label{ct1}
St^{\Delta_i}(Q_{n,s}) &= (-1)^nQ_{n,0}(P_{n,i,s}^p + R_{n,i}^pQ_{n,s}),
\end{align}
where $P_{n,i,0} = 0$, $P_{n,i,s} = -[0,\ldots,\widehat{s-1},\ldots,n-1,i-1]/L_n$, for $s > 0$ and 
$$R_{n,i} = [0,1,\ldots,n-2,i-1]/L_n.$$
\end{thm}

Note that $R_{n,i} = -P_{n,i,n}$. The case $p=2$ of \eqref{ct1} is also proved by H\uhorn ng~ \cite{hun}. It is used to explicitly compute the mod 2 Margolis homology of the Dickson algebra for any $i$. 

\medskip
We need the following results for the proof of the theorem.

\begin{prop} [See Smith-Switzer \cite{ssw}, Wilkerson \cite{wil}]\label{dlss}  For any $0\leqslant s <n$ and $1\leqslant i\leqslant n$, we have
$$St^{\Delta_i}(Q_{n,s})=\begin{cases} (-1)^{s-1}Q_{n,0},& i=s>0,\\
	(-1)^{n}Q_{n,0}Q_{n,s},&i=n,\\
	0,\qquad &\text{otherwise.}\end{cases}$$
\end{prop}
The following proposition has been proved in \cite[Prop. 1.2]{su0} for $p = 2$ and in \cite[Prop. 1.2]{su1} for $p > 2$.

\begin{prop}
\label{dl1} For any sequence  $(e_{1 },\ldots,e_n)$ of non-negative integers, we have 
$$[e_{1},\ldots , e_{n-1}, e_n+n]  =\sum_{s=0}^{n-1}(-1)^{n+s-1}[e_{1},\ldots,e_{n-1}, e_n+s]Q_{n,s}^{p^{e_n}}.$$
\end{prop}
Below is an extension of Proposition \ref{dlss}.
\begin{thm}\label{dl2} 
For any $0\leqslant s <n$ and $i \geqslant 1$, we have
$$St^{\Delta_i}(Q_{n,s}) = (-1)^n[0,1,\ldots,\hat s, \ldots,n-1,i]L_n^{p-2}.$$
\end{thm}

\begin{proof} The theorem has been proved in \cite[Thm 3.1]{su2} for $p > 2$. To make the paper self-contained, we give here a proof of it for $p$ an arbitrary prime.
	
Since $L_{n,s}=L_nQ_{n,s}$ and $St^{\Delta_i}$ is a derivation, we have
	\begin{equation} \label{ct4} St^{\Delta_i}(L_{n,s})= L_nSt^{\Delta_i}(Q_{n,s})+Q_{n,s}St^{\Delta_i}(L_{n}).
	\end{equation}
From \cite[Thm 1.1]{su0} for $p=2$ and \cite[Thm 1.1]{su1} for $p > 2$, we obtain
$$St^{\Delta_i}(L_{n,s})= \begin{cases} [i,1,2,\ldots,\hat s,\ldots,n],&s>0,\\ 0,&s=0.\end{cases}$$
In particular,  $St^{\Delta_i}(L_{n})=[i,1,2,\ldots,n-1]$. 
	
If $s=0$, then $St^{\Delta_i}(L_{n,s})=0$ and 
\begin{align*}St^{\Delta_i}(L_{n})&=[i,1,2,\ldots,n-1]\\ 
& = (-1)^{n-1}[1,2,\ldots,n-1,i].
\end{align*}
Combining \eqref{ct4}, the above equalities and the relation $Q_{n,0}=L_n^{p-1}$,  we have
\begin{align*} St^{\Delta_i}(Q_{n,0})&= -St^{\Delta_i}(L_n) Q_{n,0}/L_n\\
&=(-1)^{n}[1,2,\ldots,n-1,i]Q_{n,0}/L_n\\
&= (-1)^{n}[1,2,\ldots,n-1,i]L_n^{p-2}.
\end{align*}
Hence, the theorem holds.
	
If  $s>0$, then 
$St^{\Delta_i}(L_{n,s}) = [i,1,2,\ldots,\hat s,\ldots, n].$
So, using Proposition \ref{dl1}, we get
\begin{align*}St^{\Delta_i}(L_{n,s})  &=\sum_{t=0}^{n-1}(-1)^{n-1+t}[i,1,2,\ldots,\hat s,\ldots,n-1, t]Q_{n,t}\\
&=(-1)^{n-1} [i,1,2,\ldots,\hat s,\ldots,n-1, 0]Q_{n,0}\\
&\qquad +(-1)^{n-1+s}[i,1,2,\ldots ,\hat s,\ldots ,n-1, s]Q_{n,s}\\
&= [i,1,2,\ldots,n-1]Q_{n,s}-[i,0,1,\ldots,\hat s,\ldots,n-1]Q_{n,0}.
\end{align*}
Combining \eqref{ct4}, the above equalities and the relation $Q_{n,0}=L_n^{p-1}$, we obtain
\begin{align*}St^{\Delta_i}(Q_{n,s})&= (St^{\Delta_i}(L_{n,s})-Q_{n,s}St^{\Delta_i}(L_{n}))/L_n\\
&=-[i,0,1,2,\ldots,\hat s,\ldots,n-1]Q_{n,0}/L_n\\
&=(-1)^n[0,1,2,\ldots,\hat s,\ldots, n-1,i]L_n^{p-2}.
\end{align*}
This completes the proof of the theorem.
\end{proof}

We now prove Theorem \ref{dlc}.
\begin{proof}[Proof of Theorem \ref{dlc}] By Theorem \ref{dl2}, we have
\begin{align*}
St^{\Delta_i}(Q_{n,0}) &= (-1)^n[1,2,\ldots,n-1,i]L_n^{p-2}\\
&= (-1)^n\big([0,1,\ldots,n-2,i-1]/L_n\big)^pL_n^{2p-2}\\ 
&= (-1)^nR_{n,i}^pQ_{n,0}^2.
\end{align*}
Hence, the theorem is true for $s = 0$. 

Assume that $s > 0$.
We prove the theorem by induction on $i$. By Proposition \ref{dlss}, the theorem is true for $1 \leqslant i \leqslant n$. Suppose that $i \geqslant n$ and the theorem holds for $1,2,\ldots, i$. Using Proposition \ref{dl1}, Theorem \ref{dl2} and the inductive hypothesis, we get
\begin{align*}
St^{\Delta_{i+1}}(Q_{n,s}) &= (-1)^n[0,1,\ldots,\hat s, \ldots,n-1,i+1]L_n^{p-2}\\
&=\sum_{t=0}^{n-1}(-1)^{t-1}[0,\ldots,\hat s, \ldots ,n-1, i-n + 1 +t]Q_{n,t}^{p^{i-n+1}}L_n^{p-2}\\
&= \sum_{t=0}^{n-1}(-1)^{n+ t-1}St^{\Delta_{i-n+1+t}}(Q_{n,s})Q_{n,t}^{p^{i-n+1}}\\
&= \sum_{t=0}^{n-1}(-1)^{t-1}Q_{n,0}(P_{n,i-n + 1 + t,s}^p + R_{n,i-n+1+t}^pQ_{n,s})Q_{n,t}^{p^{i-n+1}}\\
&= (-1)^nQ_{n,0}\Big(\sum_{t=0}^{n-1}(-1)^{n + t-1}P_{n,i-n + 1 + t,s}^pQ_{n,t}^{p^{i-n+1}}\\ &\hskip2cm + \Big(\sum_{t=0}^{n-1}(-1)^{n + t-1}R_{n,i-n + 1 + t}^pQ_{n,t}^{p^{i-n+1}}\Big)Q_{n,s}\Big).
\end{align*}
Using Proposition \ref{dl1}, we have
\begin{align*}
\sum_{t=0}^{n-1}&(-1)^{n + t-1}P_{n,i-n + 1 + t,s}^pQ_{n,t}^{p^{i-n+1}}\\ &= \sum_{t=0}^{n-1}(-1)^{n + t-1}\big(-[0,\ldots,\widehat{s-1},\ldots,n-1,i-n+t]/L_n\big)^pQ_{n,t}^{p^{i-n+1}}\\
&= \Big(\Big(-\sum_{t=0}^{n-1}(-1)^{n + t-1}[0,\ldots,\widehat{s-1},\ldots,n-1,i-n+t]Q_{n,t}^{p^{i-n}}\Big)/L_n\Big)^p\\
&= \big(-[0,\ldots,\widehat{s-1},\ldots,n-1,i]/L_n\big)^p = P_{n,i+1,s}^p.
\end{align*}
By a similar computation using Proposition \ref{dl1}, we obtain
\begin{align*}
\sum_{t=0}^{n-1}&(-1)^{n + t-1}R_{n,i-n + 1 + t}^pQ_{n,t}^{p^{i-n+1}}\\ &= \sum_{t=0}^{n-1}(-1)^{n + t-1}\big([0,1,\ldots,n-2,i-n+t]/L_n\big)^pQ_{n,t}^{p^{i-n+1}}\\
&= \Big(\Big(\sum_{t=0}^{n-1}(-1)^{n + t-1}[0,1,\ldots,n-2,i-n+t]Q_{n,t}^{p^{i-n}}\Big)/L_n\Big)^p\\
&= \big([0,1,\ldots,n-2,i]/L_n\big)^p = R_{n,i+1}^p.
\end{align*}
Thus, the theorem is true for $i+1$. So, the proof is completed.
\end{proof}

Using Theorem \ref{dlc} and Proposition \ref{dl1}, we can explicitly compute the action of $St^{\Delta_i}$ on the Dickson invariants $Q_{n,s}$ for $i > n$ by explicitly computing $P_{n,i,s}$ and $R_{n,i}$. The cases $i = n + 1,\, n+2$ have been computed in \cite{su2} by using Theorem \ref{dl2}.
\begin{corl}[See \cite{su2}] For $0\leqslant s <n$, we have
\begin{align*} 
St^{\Delta_{n+1}}(Q_{n,s}) &= (-1)^nQ_{n,0}(-Q_{n,s-1}^p + Q_{n,n-1}^{p}Q_{n,s}),\\
St^{\Delta_{n+2}}(Q_{n,s}) &= (-1)^nQ_{n,0}\big(Q_{n,s-2}^{p^2} - Q_{n,s-1}^pQ_{n,n-1}^{p^2} + (Q_{n,n-1}^{p^2+p} - Q_{n,n-2}^{p^2})Q_{n,s}\big).
\end{align*} 
\end{corl}

By a direct calculation using Proposition \ref{dl1}, we easily obtain the following.
\begin{corl} For $0\leqslant s <n$, we have
$$St^{\Delta_{n+3}}(Q_{n,s}) = (-1)^nQ_{n,0}(P_{n,n+3,s}^p + R_{n,n+3}^pQ_{n,s}),$$ 
where
\begin{align*}
P_{n,n+3,s}&=-Q_{n,s-3}^{p^2}+Q_{n,s-2}^pQ_{n,n-1}^{p^2}+Q_{n,s-1}Q_{n,n-2}^{p^2}-Q_{n,s-1}Q_{n,n-1}^{p^2+p},\\
R_{n,n+3}&=Q_{n,n-3}^{p^2}-Q_{n,n-2}^{p^2}Q_{n,n-1}-Q_{n,n-2}^pQ_{n,n-1}^{p^2}+Q_{n,n-1}^{p^2+p+1}.
\end{align*}
\end{corl}

\begin{corl} For any $0\leqslant s <n$ and $i \geqslant 1$, we have
\begin{align*}
St^{\Delta_{i}}(Q_{n,0}^{p-1}Q_{n,s}) = (-1)^nQ_{n,0}^{p}P_{n,i,s}^p 
\in \mathbb \mbox{\rm Ker}(St^{\Delta_{i}}).
\end{align*}
\end{corl}
\begin{proof} Since $St^{\Delta_i}$ is a derivation, using Theorem \ref{dlc}, we get
\begin{align*}
St^{\Delta_{i}}(Q_{n,0}^{p-1}Q_{n,s}) &=  (p-1)Q_{n,0}^{p-2}St^{\Delta_{i}}(Q_{n,0})Q_{n,s} + Q_{n,0}^{p-1}St^{\Delta_{i}}(Q_{n,s})\\
&= (-1)^n\big(- Q_{n,0}^{p}R_{n,i}^pQ_{n,s} + Q_{n,0}^{p}(P_{n,i,s}^p + R_{n,i}^pQ_{n,s})\big)\\
&= (-1)^nQ_{n,0}^{p}P_{n,i,s}^p.  
\end{align*}
The corollary is proved.
\end{proof}

\bigskip
{}

\end{document}